\documentclass{amsart}
\usepackage{graphicx,amssymb,amscd,graphics,verbatim,mathrsfs,latexsym,amsmath,amsthm,eucal}
\usepackage{latexsym}
\usepackage{color}
\usepackage{stmaryrd} 
\usepackage{amsfonts}
\usepackage{amssymb, mathrsfs, amsfonts, amsmath}
\usepackage{amsbsy}
\usepackage{amsfonts}
\newtheorem{definition}{Definition}
\newtheorem{lemma}{Lemma}

\newtheorem{remark}{Remark}
\newtheorem{theorem}{Theorem}

\numberwithin{equation}{section}

\theoremstyle{definition}

\numberwithin{equation}{section}

\title[Weakly Einstein critical metrics]{Weakly Einstein critical metrics of the volume functional on compact manifolds with boundary}

\author{H. Baltazar}
\author{A. Da Silva}
\author{F. Oliveira$^1$}

\address[H. Baltazar]{Departamento de Matem\'{a}tica, Universidade Federal do Piau\'{\i}\\
64049-550 Te\-re\-si\-na, Piau\'{\i}, Brazil.}
\email{halyson@ufpi.edu.br}

\address[A. da Silva]{Faculdade de Matem\'{a}tica, Universidade Federal do Par\'{a}\\
	66075-110 Be\-l\'{e}m, Par\'{a}, Brazil.}
\email{adamsilva@ufpa.br}

\address[F. Oliveira]{Departamento de Ci\^{e}ncias Naturais, Matem\'{a}tica e Estat\'{i}stica(DCME), Universidade Federal Rural do Semi-\'Arido\\
59625-900 - Mossor\'o, Rio Grande do Norte, Brazil.}
\email{fabricio@ufersa.edu.br}
\thanks{$^{1}$ Partially supported by CNPq-Brazil}


\begin{document}

\newcommand{\spacing}[1]{\renewcommand{\baselinestretch}{#1}\large\normalsize}
\spacing{1.2}

\begin{abstract}
The goal of this paper is to study weakly Einstein critical metrics of the volume functional on a compact manifold $M$ with smooth boundary $\partial M$. Here, we will give the complete classification for an $n$-dimensional, $n=3$ or $4,$ weakly Einstein critical metric of the volume functional with nonnegative scalar curvature. Moreover, in the higher dimensional case ($n\geq5$), we will established a similar result for weakly Einstein critical metric under a suitable constraint on the Weyl tensor.
\end{abstract}

\maketitle

\section{Introduction}
\label{intro}

Let $M^{n}$ ($n\geq3$) be a connected, compact n-dimensional manifold with smooth boundary $\partial M$ with a fixed metric $\gamma.$ In 2009, Miao and Tam \cite{miaotamTAMS} proved that if the first Dirichlet eigenvalue of $(n-1)\Delta_{g}+R$ on $M$ is positive, then $g$ is a critical point of the volume functional restricted to the subset of Riemannian metrics with constant scalar curvature $R$ and prescribed boundary metric $\gamma$ if and only if there is a smooth function $f$ on $M$ such that  $f|_{\partial M}=0$, and satisfying the following system of PDE's
\begin{equation}\label{eqMiaoTam1}
-(\Delta_{g} f)g+Hess_{g} f-fRic_{g}=g,
\end{equation}
where $Ric_{g}$ and $Hess_{g} f$ stand, respectively, for the Ricci tensor and the Hessian of $f$ associated to $g$ on $M^n$ (see \cite{miaotam} for more details). Hence, following the terminology used in \cite{baltazar,BDR,BDRR,miaotamTAMS}  we recall the definition of Miao-Tam critical metric.

\begin{definition}\label{eq:miaotam}
A Miao-Tam critical metric is a $3$-tuple $(M^n,\,g,\,f),$ where $(M^{n},\,g)$ is a compact Riemannian manifold of dimension at least three, with a smooth boundary $\partial M$ and $f: M^{n}\to \Bbb{R}$ is a nonnegative smooth function satisfying equation $(\ref{eqMiaoTam1})$ such that $f|_{\partial M}=0$.
\end{definition}
It is important to recall that trivial examples of Miao-Tam critical metrics are geodesic balls in simply connected space forms. Furthermore, others explicit examples are warped product that include, for instance, the spatial Schwarzschild as well as the AdS-Schwarzschild metrics (see Corollaries 3.1 and 3.2 in \cite{miaotamTAMS} for more details).

The study of a such critical metric has been subject of interest for many authors. For instance, inspired by the work of Kobayashi and Obata \cite{kobayashi,obata},  Miao and Tam in \cite{miaotamTAMS}, studied these critical metrics under Einstein and locally conformally flat conditions. For the Einstein case, they were able to prove that critical points of the volume functional are precisely the geodesic balls in  simply connected space forms. In fact, the authors obtained the following result.

\begin{theorem}[Miao-Tam, \cite{miaotamTAMS}]
\label{thmMT}
Let $(M^{n},g,f)$ be a connected, compact Einstein Miao-Tam critical metric with smooth boundary $\partial M.$ Then $(M^{n},g)$ is isometric to a geodesic ball in a simply connected space form $\mathbb{R}^{n}$, $\mathbb{H}^{n}$ or $\mathbb{S}^{n}.$
\end{theorem}
On the other hand, for the locally conformally flat condition, the authors first constructed explicit examples of critical metrics which are in the form of warped products. To be more precise, if $M$ has disconnected boundary then, $\partial M$ has exactly two connected components, and $(M,g)$ is isometric to $(I\times N,dt^{2}+r^{2}g_{N}),$ where $I$ is a finite interval in $\mathbb{R}$ containing the origin $0,$ $(N,g_{N})$ is a compact without boundary manifold with constant sectional curvature $\kappa_{0},$ $r$ is a positive function on $I$ satisfying $r'(0)=0$ and
$$r''+\frac{R}{n(n-1)}r=ar^{1-n}$$
for some constant $a>0,$ and the constant $\kappa_{0}$ satisfies
$$(r')^{2}+\frac{R}{n(n-1)}r^{2}+\frac{2a}{n-2}r^{2-n}=\kappa_{0}.$$
For more details on this result see \cite[Theorem 1.2]{miaotamTAMS}.

Later on, Barros, Di\'ogenes, and Ribeiro  classified in \cite{BDR} these critical metrics for a $4$-dimensional simply connected manifold under Bach-flat assumption. Such result was improved by the first author, Batista, and Bezerra for a generic manifold in an arbitrary dimension (see \cite[Theorem 2]{BBB}). Meanwhile, Barros, and Da Silva presented in \cite{Adam}, an upper bound for the area of the boundary of a compact $n$-dimensional oriented Miao-Tam critical metric (see also \cite{BLF2016,BDRR}). For more references on the critical metrics of the volume functional, see \cite{baltazar,balt17,Adam,BaltRR,BDRR,Kim,Yuan}, and references therein.

In order to proceed, it is fundamental to remember that a Riemannian manifold is said to be weakly Einstein if the Riemannian curvature operator $Rm$ satisfies the following identity

\begin{equation}\label{eqwEinstein}
\breve{R}_{ij}=\frac{|Rm|^{2}}{n}g_{ij},
\end{equation}
where the $2$-tensor $\breve{R}_{ij}$ is defined by $\breve{R}_{ij}=R_{ipqr}R_{jpqr}.$ Such concept was introduced by Euh, Park and Sekigawa in \cite{eps} which was inspired by that of a super-Einstein manifold as defined in \cite{GW}.

Let us point out that weakly Einstein manifolds appear as critical points of a well-known quadratic functional on compact manifold without boundary. More precisely, Catino \cite{catino15} studied critical points of the functional
$$\mathcal{F}_{t,s}(g)=\int|Ric|^{2}dM+t\int R^{2}dM+s\int|Rm|^{2}dM, \qquad t,s\in\mathbb{R},$$
restricted to the space of equivalence classes of a smooth Riemannian metric with unit volume on $M$, denoted by $\mathcal{M}_{1}$. In this case, the author concluded that an Einstein metric is critical to the functional $\mathcal{F}_{t,s}$ on $\mathcal{M}_{1}$ if and only if the manifold $M$ is weakly Einstein. Moreover, considering a four dimensional manifold, it is not difficult to verify that Einstein manifolds are in fact weakly Einstein (cf. Remark~\ref{4DEinstein} in Section~\ref{Preliminaries}). We refer the readers to \cite{to15,eps,eps13,eps11,GW,HY16} for more results on this subject.

Here, we shall replace the assumption of Einstein in the Miao Tam result (see Theorem~\ref{thmMT}) by the weakly Einstein condition, which is weaker that the former one in low dimension. We now state our first result.

\begin{theorem}\label{thmMainA}
Let $(M^3,\,g,\,f)$ be a weakly Einstein Miao-Tam critical metric with nonnegative scalar curvature. Then $(M^{3},g)$ is isometric to a geodesic ball in a simply connected space form $\mathbb{R}^{3}$ or $\mathbb{S}^{3}.$
\end{theorem}

Before presenting our second result let us remember that the Riemannian product $\mathbb{S}^{2}\times\mathbb{H}^{2}$ with their standard metrics is weakly Einstein, but not Einstein. In addition, an $n$-dimensional, $n\geq5,$ Einstein metric does not necessarily is weakly Einstein, since the right hand side of Eq. (\ref{eqweakly}) will not necessarily null. After these considerations, we have the following theorem for a 4-dimensional weakly Einstein Miao-Tam critical metric.

\begin{theorem}\label{thmMainB}
Let $(M^4,\,g,\,f)$ be a weakly Einstein Miao-Tam critical metric with nonnegative scalar curvature. Then $(M^{4},g)$ is isometric to one of the following critical metrics:
\begin{itemize}
\item[$(i)$] The geodesic ball in a simply connected space form $\mathbb{R}^{4}$ or $\mathbb{S}^{4}.$

\item[$(ii)$]The standard cylinder over a compact without boundary manifold $(N,g_{N})$  with constant sectional curvature $\kappa_{0}>0$ with the product metric
$$\left(I\times N,dt^{2}+\frac{\kappa_{0}^{2}t^{2}+a}{\kappa_{0}}g_{N}\right),$$
where $I$ is a finite interval in $\mathbb{R}$ containing the origin $0$ and for some constant $a>0.$

\item[$(iii)$] $(M^{4},g)$ is covered by the above mentioned warped product in $(ii)$ with a covering group $\mathbb{Z}_{2}.$
\end{itemize}
\end{theorem}

Based on the previous result, it is natural to ask what occurs in high dimension. In this sense, we shall prove a rigidity result for an $n$-dimensional Miao-Tam critical metric satisfying the weakly Einstein assumption and a suitable restriction on the Weyl tensor. More precisely, we have the following result.

\begin{theorem}\label{thmMainC}
Let $(M^n,\,g,\,f),$ $n\geq5,$ be a weakly Einstein Miao-Tam critical metric with nonpositive scalar curvature and  such that the Weyl tensor satisfies $W|_{T\partial M}=0.$ Then $(M^{n},g)$ is isometric to a geodesic ball in a simply connected space form $\mathbb{R}^{n}$  or $\mathbb{H}^{n}.$
\end{theorem}

\section{Background}
\label{Preliminaries}

In this section we shall collect some basic results that will be useful in the proof of our main result. We begin by recalling some special tensors as well as some terminology in the study of curvature for a Riemannian manifold $(M^n,\,g),\,n\ge 3.$ The Weyl tensor $W$ is defined by the following decomposition formula
\begin{eqnarray}
\label{weyl}
R_{ijkl}&=&W_{ijkl}+\frac{1}{n-2}\big(R_{ik}g_{jl}+R_{jl}g_{ik}-R_{il}g_{jk}-R_{jk}g_{il}\big) \nonumber\\
 &&-\frac{R}{(n-1)(n-2)}\big(g_{jl}g_{ik}-g_{il}g_{jk}\big),
\end{eqnarray}
where $R_{ijkl}$ stands for the Riemannian curvature tensor, while the Cotton tensor $C$ is defined as follows
\begin{equation}
\label{cotton}
C_{ijk}=\nabla_{i}R_{jk}-\nabla_{j}R_{ik}-\frac{1}{2(n-1)}\big(\nabla_{i}R g_{jk}-\nabla_{j}Rg_{ik}).
\end{equation}
When $n\geq 4$ we have
\begin{equation}
\label{cottonwyel} C_{ijk}=-\frac{(n-2)}{(n-3)}\nabla_{l}W_{ijkl},
\end{equation}
for more details about these tensors we refer to \cite{besse}.

An important remark about the Cotton tensor $C_{ijk}$ is that it is skew-symmetric in the first two indexes and
trace-free in any index, that is,
\begin{equation}
\label{cottonprop} C_{ijk}=-C_{jik} \,\,\,\,\,\, {\rm and} \,\,\,\,\,\, g^{ij}C_{ijk}=g^{ik}C_{ijk}=0.
\end{equation}

Now, we recall that for operators $S,T:\mathcal{H} \to \mathcal{H}$ defined over an $n$-dimensional Hilbert space $\mathcal{H}$ the Hilbert-Schmidt inner product is defined according to
\begin{equation}
\langle S,T \rangle =\rm tr\big(ST^{*}\big), \label{inner}
\end{equation}
where $\rm tr$ and $*$ denote, respectively, the trace and the adjoint operation. Moreover, if $I$ denotes the identity operator on $\mathcal{H}$ the traceless operator of $T$ is given by
\begin{equation}
\label{eqtr1}
\mathring{T}=T - \frac{\rm tr T}{n}I.
\end{equation}

With this notation in mind, the following curvature expression can be deduced directly from (\ref{weyl}) and it will be quite useful for our purpose.

\begin{lemma}\label{WEaux}
Let $(M^{n},g)$ be a Riemannian manifold. Then the following identity holds:
\begin{eqnarray}\label{eqweakly}
\breve{R}_{ij}-\frac{|Rm|^{2}}{n}g_{ij}&=&\breve{W}_{ij}-\frac{|W|^{2}}{n}g_{ij}+\frac{4}{n-2}W_{ipjq}\mathring{R}_{pq}+\frac{2(n-4)}{(n-2)^{2}}\mathring{R}_{iq}\mathring{R}_{qj}\nonumber\\
&&+\frac{4R}{n(n-1)}\mathring{R}_{ij}-\frac{2(n-4)}{n(n-2)^{2}}|\mathring{Ric}|^{2}g_{ij}.
\end{eqnarray}
Here, $\breve{W}_{ij}$ is defined similarly to $\breve{R}_{ij}$ in $(\ref{eqwEinstein})$.

\end{lemma}

In the sequel, it is worth reporting here the following interesting formula involving the Weyl
tensor in dimension four. We refer readers to \cite{derd1} for its proof.

\begin{lemma}\label{Dedaux}
Let $(M^{4},g)$ be a four-dimensional Riemannian manifold. Then the following identity holds:
\begin{eqnarray*}
\breve{W}_{ij}=\frac{|W|^{2}}{4}g_{ij}.
\end{eqnarray*}
\end{lemma}

\begin{remark}\label{4DEinstein}
As a consequence of these two previous lemmas is that a four-dimensional Einstein metric is immediately weakly Einstein.
\end{remark}

Next, we remember that the fundamental equation of Miao-Tam critical metric is given by
\begin{equation}
\label{eqfund1} -(\Delta f)g+Hess f-fRic=g.
\end{equation}
Taking the trace of Eq. (\ref{eqfund1}) we arrive at
\begin{equation}
\label{eqtrace} \Delta f +\frac{fR+n}{n-1}=0.
\end{equation}

By using (\ref{eqtrace}) we check immediately that
\begin{equation}
\label{IdRicHess} f\mathring{Ric}=\mathring{Hess f}.
\end{equation}

To fix notations, for a Miao-Tam critical metric, we recall the following 3-tensor defined in \cite{BDR},
\begin{eqnarray}\label{TensorT}
T_{ijk}&=&\frac{n-1}{n-2}(R_{ik}\nabla_{j}f-R_{jk}\nabla_{i}f)-\frac{R}{n-2}(g_{ik}\nabla_{j}f-g_{jk}\nabla_{i}f)\nonumber\\
&&+\frac{1}{n-2}(g_{ik}R_{js}\nabla_{s}f-g_{jk}R_{is}\nabla_{s}f).
\end{eqnarray}
Note that, $T_{ijk}$ has the same symmetry properties as the Cotton tensor:
$$T_{ijk}=-T_{jik}\;\;\;and\;\;\;g^{ij}T_{ijk}=g^{ik}T_{ijk}=0.$$
Furthermore, we remember that the tensor $T_{ijk}$ is related with to the Cotton tensor $C_{ijk}$ as well as the Weyl tensor $W_{ijkl}$ by
\begin{equation}\label{CTW}
fC_{ijk}=T_{ijk}+W_{ijkl}\nabla_{l}f.
\end{equation}

To close this section, let us recall a result that can be found in \cite[Lemma 5]{BDRR}.

\begin{lemma}\label{intauxRIC}
Let $(M^{n},g,f)$ be a Miao-Tam critical metric. Then we have:
$$\int_{M}f|\mathring{Ric}|^{2}dM+\frac{1}{|\nabla f|}\int_{\partial M}\mathring{Ric}(\nabla f,\nabla f)d\sigma=0.$$
\end{lemma}

\section{Weakly Einstein Miao-Tam critical metrics}

The aim of this section is to prove the rigidity of the Miao-Tam critical metrics under assumption that our Riemannian manifold is weakly Einstein. In order to prove our theorems, we shall start showing that at a point $q\in \partial M$, $\nabla f$ is an eigenvector of the Ricci tensor. More precisely, we have the following result.

\begin{lemma}\label{autoRIC}
Let $(M^{n},g,f)$ be a Miao-Tam critical metric. Then $\nabla f$ is an eigenvector to the Ricci tensor for all $q\in\partial M$.
\end{lemma}
\begin{proof}
In fact,  we consider  an orthonormal frame $\{e_{1},\ldots,e_{n}\}$ diagonalizing $Ric$ at a point $q\in \partial M^{n}$ such that $\nabla f(q)\neq0$, namely, we have that $Ric(e_{i})=\alpha_{i}e_{i}$, where $\alpha_{i}$ are the associated eigenvalues.

Since we already know that $f|_{\partial M}=0,$ it follows from Eq. (\ref{CTW}) that
$$0=fC_{ijk}\nabla_{k}f=T_{ijk}\nabla_{k}f.$$
Thus, by definition of the tensor $T_{ijk},$ see Eq. (\ref{TensorT}), we get
$$R_{jk}\nabla_{k}f\nabla_{i}f-R_{ik}\nabla_{k}f\nabla_{j}f=0,$$
and consequently, computing this last expression at $q\in \partial M,$ we deduce
\begin{equation}\label{lambda}
(\alpha_{j}-\alpha_{i})\nabla_{i}f\nabla_{j}f=0.
\end{equation}
Hence, if we consider the following nonempty set $L=\{i;\nabla_{i}f\neq0\},$ then from (\ref{lambda}) we have $\alpha_{i}=\alpha$, for all $i\in L,$ and therefore
$$Ric(\nabla f)=Ric\Big(\displaystyle\sum_{i\in L}\nabla_{i}fe_{i}\Big)=\sum_{i\in L}\nabla_{i}f\alpha_{i}e_{i}=\alpha\nabla f.$$
This is what we want to prove.
\end{proof}

In what follows, let $q\in M^{n}$ be an arbitrary point on $\partial M$ and consider an orthonormal frame  $\{e_{1},e_{2},\ldots, e_{n}\}$ in a neighborhood of $q$ contained in $M$ such that $e_{1}=-\frac{\nabla f}{|\nabla f|}.$

In the sequel, we obtain an important expression in order to prove our results. More precisely, we shall compute the norm of the auxiliary tensor $T_{ijk}$ on $\partial M$ of a Miao-Tam critical metric.

\begin{lemma}\label{NormaTB}
Let $(M^n,g,f)$ be a Miao-Tam critical metric. Then, restrict to the boundary $\partial M,$ the following identity holds
$$|T_{ijk}|^{2}=\frac{2(n-1)^{2}}{(n-2)^{2}}  |\mathring{Ric}|^{2}|\nabla f|^{2} -\frac{2n(n-1)}{(n-2)^{2}}\mathring{R}_{11}^{2}|\nabla f|^{2}.$$
\end{lemma}
\begin{proof}We start the proof by using Eq. (\ref{TensorT}) to arrive at
\begin{eqnarray*}
|T_{ijk}|^{2}&=&\frac{2(n-1)^{2}}{(n-2)^{2}}|Ric|^{2}|\nabla f|^{2}-\frac{2n(n-1)}{(n-2)^{2}}|Ric(\nabla f)|^{2}\\
&&+\frac{4(n-1)}{(n-2)^{2}}RRic(\nabla f,\nabla f)-\frac{2(n-1)}{(n-2)^{2}}R^{2}|\nabla f|^{2}\\
&=&\frac{2(n-1)^{2}}{(n-2)^{2}}|\mathring{Ric}|^{2}|\nabla f|^{2}-\frac{2n(n-1)}{(n-2)^{2}}|Ric(\nabla f)|^{2}\\
&&+\frac{4(n-1)}{(n-2)^{2}}R\mathring{Ric}(\nabla f,\nabla f)+\frac{2(n-1)}{n(n-2)^{2}}R^{2}|\nabla f|^{2}.
\end{eqnarray*}
Now, from Lemma~\ref{autoRIC} we have that $R_{1a}=0$ for any $2\leq a\leq n$, and consequently the above expression becomes
\begin{eqnarray*}
|T_{ijk}|^{2}&=&\frac{2(n-1)^{2}}{(n-2)^{2}}|\mathring{Ric}|^{2}|\nabla f|^{2}-\frac{2n(n-1)}{(n-2)^{2}}R^{2}_{11}|\nabla f|^{2}\\
&&+\frac{4(n-1)}{(n-2)^{2}}R\mathring{R}_{11}|\nabla f|^{2}+\frac{2(n-1)}{n(n-2)^{2}}R^{2}|\nabla f|^{2}\\
&=&\frac{2(n-1)^{2}}{(n-2)^{2}}|\mathring{Ric}|^{2}|\nabla f|^{2}-\frac{2n(n-1)}{(n-2)^{2}}\mathring{R}^{2}_{11}|\nabla f|^{2}.
\end{eqnarray*}
So, the proof is completed.
\end{proof}

Proceeding, we shall verify that the same identity obtained in \cite[Lemma 2.2]{HY16} is also true for a Miao-Tam critical metric. More precisely, we have:
\begin{lemma}\label{auxnormaRIC}
Let $(M,g,f)$ be a Miao-tam critical metric. Then, restrict to the boundary $\partial M,$ we have the following identity
$$|\mathring{Ric}|^{2}=\frac{n}{n-1}\mathring{R}_{11}^{2}+\left(\frac{n-2}{n-1}\right)^{2}|W^{\nu}|^{2},$$
where $W^{\nu}_{ij}=W_{i1j1}.$
\end{lemma}
\begin{proof}
From (\ref{CTW}) jointly with fact that $f|_{\partial M}=0,$ we obtain
\begin{eqnarray}\label{eqWric1}
W_{ipjq}R_{pq}\nabla_{i}f\nabla_{j}f&=&T_{ipq}R_{pq}\nabla_{i}f\nonumber\\
&=&\frac{1}{2}T_{ipq}(R_{pq}\nabla_{i}f-R_{iq}\nabla_{p}f)\nonumber\\
&=&-\frac{n-2}{2(n-1)}|T_{ijk}|^{2},
\end{eqnarray}
where in the last step we use the definition of the tensor $T_{ijk},$ see Eq. (\ref{TensorT}).

On the other hand, we notice that
\begin{eqnarray*}
|W^{\nu}|^{2}&=&W_{i1j1}W_{i1j1}\\
&=&\frac{1}{|\nabla f|}W_{i1j1}T_{i1j}\\
&=&\frac{n-1}{(n-2)|\nabla f|}W_{i1j1}(R_{ij}\nabla_{1}f-R_{1j}\nabla_{i}f),
\end{eqnarray*}
which by Lemma~\ref{autoRIC} we can infer
\begin{eqnarray}\label{eqWric2}
|W^{\nu}|^{2}&=&-\frac{n-1}{(n-2)|\nabla f|^{2}}W_{ipjq}R_{ij}\nabla_{p}f\nabla_{q}f.
\end{eqnarray}
Hence, comparing (\ref{eqWric1}) with (\ref{eqWric2}) we get
\begin{eqnarray}\label{TWnu}
|T_{ijk}|^{2}=2|\nabla f|^{2}|W^{\nu}|^{2}.
\end{eqnarray}
So, it suffices to substitute this equality in Lemma~\ref{NormaTB} to conclude the proof of the lemma.
\end{proof}

Now, through the weakly Einstein condition, we can also show the next lemma.

 \begin{lemma}\label{EqRIC11}
Let $(M^{n},g,f)$ be a weakly Einstein Miao-Tam critical metric. Then, restrict to the boundary $\partial M,$ we have the following identity
$$\frac{4R}{n(n-1)}\mathring{R}_{11}=\frac{1}{n}|W|^{2}-\frac{2(n^{2}-2n-2)}{n(n-2)}|\mathring{Ric}|^{2}+2\mathring{R}^{2}_{11}.$$
\end{lemma}
\begin{proof}
First of all, by Lemma~\ref{WEaux} and our assumption that $M^{n}$ satisfies the weakly Einstein condition, we obtain
\begin{eqnarray*}
0&=&W_{ipqr}W_{jpqr}\nabla_{i}f\nabla_{j}f-\frac{|W|^{2}}{n}|\nabla f|^{2}+\frac{4}{n-2}W_{ipjq}\mathring{R}_{pq}\nabla_{i}f\nabla_{j}f\\
&&+2\frac{(n-4)}{(n-2)^{2}}\mathring{R}_{1q}\mathring{R}_{q1}|\nabla f|^{2}+\frac{4R}{n(n-1)}\mathring{R}_{11}|\nabla f|^{2}-\frac{2(n-4)}{n(n-2)^{2}}|\mathring{Ric}|^{2}|\nabla f|^{2},
\end{eqnarray*}
which can be rewritten, using Eq. (\ref{CTW}) and Lemma~\ref{autoRIC}, as
\begin{eqnarray*}
0&=&\frac{1}{|\nabla f|^{2}}|T_{ijk}|^{2}-\frac{|W|^{2}}{n}+\frac{4}{n-2}W_{1p1q}\mathring{R}_{pq}+2\frac{(n-4)}{(n-2)^{2}}\mathring{R}^{2}_{11}\\
&&+\frac{4R}{n(n-1)}\mathring{R}_{11}-\frac{2(n-4)}{n(n-2)^{2}}|\mathring{Ric}|^{2}.
\end{eqnarray*}
Now, by equation (\ref{eqWric1}) jointly with Lemma~\ref{NormaTB}, it is easy to verify that the above identity becomes
\begin{eqnarray*}
\frac{4R}{n(n-1)}\mathring{R}_{11}&=&-\frac{n-3}{(n-1)|\nabla f|^{2}}|T_{ijk}|^{2}+\frac{|W|^{2}}{n}-2\frac{(n-4)}{(n-2)^{2}}\mathring{R}_{11}^{2}+\frac{2(n-4)}{n(n-2)^{2}}|\mathring{Ric}|^{2}\\
&=&\frac{1}{n}|W|^{2}-\frac{2(n^{2}-2n-2)}{n(n-2)}|\mathring{Ric}|^{2}+2\mathring{R}^{2}_{11},
\end{eqnarray*}
as desired.
\end{proof}

\subsection{Proof of Theorems \ref{thmMainA} and \ref{thmMainB} }
\begin{proof}

It is well known that in the three-dimensional case the Weyl tensor vanishes completely. Hence, when we restrict to this case it is not difficult to see that, substituting Lemma~\ref{auxnormaRIC} into Lemma~\ref{EqRIC11}, we obtain
\begin{equation}\label{3Dcase}
R\mathring{R}_{11}=|\mathring{Ric}|^{2}.
\end{equation}
Furthermore, taking into account that the scalar curvature is nonnegative, it follows from (\ref{3Dcase}) that $\mathring{R}_{11}\geq0.$ Then, we are in position to use Lemma~\ref{intauxRIC} to infer that $(M^n ,\,g)$ is a Einstein manifold, and by Theorem~\ref{thmMT}, we may conclude that $(M^n ,\,g)$ is isometric to a geodesic ball in a simply connected space form $\Bbb{R}^{3}$ or $\Bbb{S}^{3}$, which proves Theorem \ref{thmMainA}.

In what follows we shall consider the indexes $a,b,c,d\in\{2,\dots,n\}.$ Now, we claim that
\begin{eqnarray}\label{normW}
|W|^{2}=4|W^{\nu}|^{2}+W_{abcd}W_{abcd},
\end{eqnarray}
on the boundary $\partial M.$ In fact, using (\ref{eqWric1}) and (\ref{TWnu})  we obtain
\begin{eqnarray}\label{eqnormaW}
|W|^{2}&=&W_{ijkl}W_{ijkl}\nonumber\\
&=&W_{ijk1}W_{ijk1}+W_{ijka}W_{ijka}\nonumber\\
&=&\frac{1}{|\nabla f|^{2}}|T_{ijk}|^{2}+W_{ijka}W_{ijka}\nonumber\\
&=&2|W^{\nu}|^{2}+W_{1jka}W_{1jka}+W_{djka}W_{djka}.
\end{eqnarray}
In order to proceed, noticing that $W_{1jka}=\frac{1}{|\nabla f|}T_{akj}$ we use the symmetries of Weyl tensor to arrive at
\begin{eqnarray*}
W_{1jka}&=&\frac{1}{|\nabla f|}\left\{\frac{n-1}{n-2}R_{aj}\nabla_{k}f-\frac{R}{n-2}g_{aj}\nabla_{k}f+\frac{1}{n-2}g_{aj}R_{ks}\nabla_{s}f\right\}.
\end{eqnarray*}
Thus, analyzing each possibility in the above expression, it is immediate to check that
this data substituted in (\ref{eqnormaW}) becomes
\begin{eqnarray*}
|W|^{2}&=&3|W^{\nu}|^{2}+W_{a1jd}W_{a1jd}+W_{abjd}W_{abjd}\\
&=&4|W^{\nu}|^{2}+W_{a1cd}W_{a1cd}+W_{ab1d}W_{ab1d}+W_{abcd}W_{abcd},
\end{eqnarray*}
and since we already known that $W_{abc1}=0,$ clearly we obtain Identity (\ref{normW}).

Next, with a straightforward computation using Lemma~\ref{auxnormaRIC} and Lemma~\ref{EqRIC11} we can deduce the following identity
\begin{eqnarray}\label{R11 AUX}
\frac{4R}{n(n-1)}\mathring{R}_{11}&=&\frac{1}{n}|W|^{2}-\frac{2(n^{2}-2n-2)}{n(n-2)}|\mathring{Ric}|^{2}\nonumber\\
&&+2\left(\frac{n-1}{n}|\mathring{Ric}|^{2}-\frac{(n-2)^{2}}{n(n-1)}|W^{\nu}|^{2}\right)\nonumber\\
&=&\frac{1}{n}|W|^{2}-\frac{2(n-2)^{2}}{n(n-1)}|W^{\nu}|^{2}-\frac{2(n-4)}{n(n-2)}|\mathring{Ric}|^{2},
\end{eqnarray}
which from Eq. (\ref{normW}) allow us to conclude a key inequality in the four-dimensional case, i.e.,
 $$R\mathring{R}_{11}\geq|W^{\nu}|^{2}.$$
So, when scalar curvature is positive, we get $\mathring{R}_{11}\geq0.$ Then, the result follows from Lemma~\ref{intauxRIC} combined with Theorem~\ref{thmMT}.

Now we shall restrict to a $4$-dimensional weakly Einstein critical metric with $R=0.$ With this consideration in mind, it follows from Lemma \ref{WEaux} and Lemma~\ref{Dedaux} that
\begin{equation}\label{WR}
W_{ikjl}R_{kl}=0.
\end{equation}
Consequently, taking the derivative in the above expression, and using Eq. (\ref{cottonwyel}), we obtain
\begin{eqnarray}\label{wraux}
0&=&\nabla_{i}W_{ikjl}R_{kl}+W_{ikjl}\nabla_{i}R_{kl}\nonumber\\
&=&-\frac{1}{2}C_{ljk}R_{kl}+\frac{1}{2}W_{ikjl}C_{ikl}.
\end{eqnarray}
Then, computing Identity (\ref{wraux}) in $\nabla f$, and changing the index for simplicity, we deduce
\begin{eqnarray*}
0&=&C_{ijk}R_{jk}\nabla_{i}f-W_{ijkl}\nabla_{l}fC_{ijk}\\
&=&\frac{1}{2}C_{ijk}(R_{jk}\nabla_{i}f-R_{ik}\nabla_{j}f)-W_{ijkl}\nabla_{l}fC_{ijk},
\end{eqnarray*}
that can be rewritten, using (\ref{TensorT}), as
\begin{equation}\label{keyC}
\frac{1}{3}C_{ijk}T_{ijk}+W_{ijkl}\nabla_{l}fC_{ijk}=0.
\end{equation}
In order to proceed, by (\ref{WR}) and definition of the auxiliary tensor $T_{ijk}$ we have that
$$T_{ijk}W_{ijkl}\nabla_{l}f=\frac{3}{2}R_{ik}\nabla_{j}fW_{ijkl}\nabla_{l}f=0,$$
which combining with (\ref{CTW}) allow us to conclude two key identities, i.e.,
$$fC_{ijk}T_{ijk}=|T|^{2},$$
and
$$fC_{ijk}W_{ijkl}\nabla_{l}f=|\iota_{\nabla f}W|^{2},$$
where $\iota$ denotes the interior product $(\iota_{\nabla f}W)_{ijk}=W_{ijkl}\nabla_{l}f.$
So, returning to equation (\ref{keyC}), it is easy to see that $T_{ijk}=W_{ijkl}\nabla_{l}f=0,$ and from Lemma~\ref{Dedaux} we can infer
\begin{equation}\label{normaW}
0=W_{ipqr}W_{jpqr}\nabla_{i}f\nabla_{j}f=\frac{|W|^{2}}{4}|\nabla f|^{2}.
\end{equation}
Now, we choose harmonic coordinates to deduce that the potential function $f$ and the metric $g$ are analytic, see for example
\cite[Proposition 2.8]{Corvino} (or \cite[Proposition 2.1]{CEM}). Whence, equation \eqref{normaW} yields that $M^4$ is locally conformally flat, since $f$ can not be constant in a non-empty open set. Therefore, we can use the Theorem 1.2 of \cite{miaotamTAMS} to conclude that $(M^{4},g)$ is either isometric to a geodesic ball in a simply connected space form $\Bbb{R}^{4}$, or isometric to a standard cylinder over a compact without boundary manifold $(N,g_{N})$  with constant sectional curvature $\kappa_{0}>0$ with the product metric $$\left(I\times N,dt^{2}+\frac{\kappa_{0}^{2}t^{2}+a}{\kappa_{0}}g_{N}\right),$$
where $I$ is a finite interval in $\mathbb{R}$ containing the origin $0$ and for some constant $a>0$, or $(M^{4},g)$ is covered by the such warped product with a covering group $\mathbb{Z}_{2}.$ So, the proof is completed.

\subsection{Proof of Theorem \ref{thmMainC}}
Since we are assuming $W|_{T\partial M}=0,$ it suffices to substitute (\ref{normW}) in (\ref{R11 AUX}), to deduce
$$\frac{2R}{n-1}\mathring{R}_{11}=-\frac{n^{2}-6n+6}{n-1}|W^{\nu}|^{2}-\frac{n-4}{n-2}|\mathring{Ric}|^{2}.$$
Thus, as the Riemannian manifold $M^{n}$ has non-positive scalar curvature, and $n\geq 5$, it is immediate to verify that the above expression gives $\mathring{R}_{11}\geq0.$ Therefore, we are in position to invoke Lemma~\ref{intauxRIC} to deduce that $(M^n ,\,g)$ is an Einstein manifold, and, once more by Theorem~\ref{thmMT}, we may conclude that $(M^n ,\,g)$ is isometric to a geodesic ball in a simply connected space form $\Bbb{R}^{n}$ or $\Bbb{H}^{n}.$
\end{proof}


\end{document}